\documentclass[10pt, reqno]{amsart}
\usepackage{amsmath,amssymb,amsthm,graphicx}
\usepackage[left=2.5cm,right=2.2cm,top=2.5cm,bottom=3.5cm]{geometry}

\usepackage[breaklinks,hypertexnames=false]{hyperref}
\hypersetup{
	colorlinks = true, 
	urlcolor = blue, 
	linkcolor = blue, 
	citecolor = blue 
}
\usepackage{orcidlink}
\usepackage{dsfont}
\usepackage{parskip}
\makeatletter
\def\thm@space@setup{%
  \thm@preskip=\parskip \thm@postskip=0pt
}
\makeatother

\usepackage[hang,flushmargin]{footmisc}
\usepackage{footnote} 
\makesavenoteenv{tabular}

\usepackage{caption}
\let\le\leqslant
\let\ge\geqslant

\let\eps\varepsilon

\let\mc\mathcal

\newcommand{\ZZ}{\mathbb{Z}}
\newcommand{\RR}{\mathbb{R}}
\newcommand{\NN}{\mathbb{N}}

\newcommand{\EE}{\mathbb{E}}
\usepackage{thmtools}
\usepackage[noabbrev,capitalize,nameinlink,nosort]{cleveref}

\newtheorem{lemma}{Lemma}[section]
\newtheorem{proposition}[lemma]{Proposition}
\newtheorem{theorem}[lemma]{Theorem}

\crefname{fact}{Fact}{Facts}

\theoremstyle{definition}
\newtheorem{definition}[lemma]{Definition}

\crefname{claim}{Claim}{Claims}
\newtheorem*{remark*}{Remark}

\DeclareMathOperator{\Ber}{Ber}

\newcommand{\prob}[1]{\mathbb{P}\left[#1\right]}

\let\originalleft\left
\let\originalright\right
\renewcommand{\left}{\mathopen{}\mathclose\bgroup\originalleft}
\renewcommand{\right}{\aftergroup\egroup\originalright}

\crefformat{equation}{#2(#1)#3}

\title[]{\huge N\lowercase{o-}$(k+1)$\lowercase{-in-line problem for large constant }$k$}
\author[]{\Large A\lowercase{lexandr }G\lowercase{rebennikov and }M\lowercase{atthew }K\lowercase{wan}}

\address{Institute of Science and Technology Austria (ISTA)}
\email{\href{mailto:aleksandr.grebennikov@ist.ac.at}{\nolinkurl{aleksandr.grebennikov@ist.ac.at}}}
\address{Institute of Science and Technology Austria (ISTA)}
\email{\href{mailto:matthew.kwan@ist.ac.at}{\nolinkurl{matthew.kwan@ist.ac.at}}}

\thanks{Both authors are supported by ERC Starting Grant “RANDSTRUCT” No. 101076777.}

\begin{document}

\begin{abstract}
    How many points can be placed in an $n\times n$ grid so that every (affine) line contains at most $k$ points? We prove that for $n \ge k \ge 10^{37}$ the maximum number of points is exactly $kn$. Our proof builds on the recent work of Kov\'acs, Nagy, and Szab\'o (who proved an analogous result when $k$ is at least about $\sqrt{n \log n}$), incorporating ideas of Jain and Pham.
    Using the same approach, we also obtain new bounds for higher-dimensional extensions of this problem.
\end{abstract}

\maketitle

\section{Introduction}

The \emph{no-three-in-line} problem, posed by Dudeney at the beginning of the 20th century~\cite{dudeney-1917}, asks for the maximum number of points that can be placed on an $n \times n$ grid such that no three points are collinear. Despite significant attention over the years, the problem remains open: the best known upper bound $2n$ comes from the observation that each of the $n$ horizontal lines can contain at most two points, while the best known lower bound $(1.5-o(1))n$ comes from the \emph{modular hyperbola} construction of Hall, Jackson, Sudbery, and Wild~\cite{HJSW-75} (improving on an earlier algebraic construction of Erd\H os; see \cite{Rot51}). For more history and background, see for example the surveys by Brass, Moser, and Pach~\cite{brass-moser-pach-05} and Eppstein~\cite{eppstein-18}, and Green's list of open problems \cite{green-100-open-problems}.

Several different conjectures have been made concerning the asymptotic behaviour of the answer to the no-three-in-line problem, including suggestions that it could be roughly $1.5n$ \cite{green-100-open-problems}, roughly $2n$~\cite{brass-moser-pach-05}, or somewhere in between~\cite{eppstein-18,guy-kelly-68}. For small values of $n$, examples attaining the trivial upper bound $2n$ are known~\cite{anderson-79,flammenkamp-website,flammenkamp-92,flammenkamp-98}.

A natural generalisation of this problem, first studied by Brass and Knauer~\cite{brass-knauer-03} in a more general context, is to ask for the maximum size of a subset of the $n \times n$ grid containing at most $k$ points on each line. Denoting this maximum by $f_k(n)$, the trivial upper bound is $f_k(n) \le kn$. Lefmann~\cite[Proposition 2]{lefmann-12} proved that $f_k(n) = \Omega(kn)$ for every $k \ge 2$. Recently, Kov\'acs, Nagy, and Szab\'o~\cite{KNS-25} showed that $f_k(n) = kn$ for each $k \ge C \sqrt{n \log n}$ (when $C > 12.5$ and $n$ is sufficiently large in terms of $C$). They also noted~\cite[Section 4]{KNS-25} that, conditionally on some strong results about random regular graphs (i.e., the bipartite version of the Kim--Vu sandwich conjecture~\cite{KV-04}; see \cite{BIM,GIM-20,KRRS-23} for recent progress), their approach could yield an analogous statement for $k$ of order at least $\log n$.

We strengthen these results by proving that $f_k(n) = kn$ for every $n \ge k \ge K_0$, for some absolute constant $K_0$ (our proof gives $K_0 = 10^{37}$, but we did not attempt to optimise it).

\begin{theorem} \label{no-k+1-in-line}
    Let $n, k$ be integers such that $10^{37} \le k \le n$. Then there exists a subset $S$ of the $n \times n$ grid of size $kn$ such that every line contains at most $k$ points of $S$.
\end{theorem}

In a different work, Kov\'acs, Nagy, and Szab\'o \cite{KNS-25-algebraic} studied this problem in the regime when $k$ is small compared to $n$, using randomised algebraic constructions. In particular, they demonstrated that $f_k(n) \ge (k - 2 - (k \bmod 2)) n$ when $n$ is sufficiently large in terms of $k$, and conjectured that $f_k(n) = kn - o(n)$ when $k = k(n)\to \infty$ as $n \to \infty$. \cref{no-k+1-in-line} confirms this conjecture in a strong sense.

Interestingly, all known linear-size constructions for the no-three-in-line problem are algebraic in nature (though there are random constructions \cite{eppstein-blog-18,ghosal-goenka-keevash-25} of size about $n/\sqrt{\log n}$). In a related question, Green \cite{green-100-open-problems} asked whether every ``large'' no-three-in-line set reduces to an algebraic curve modulo some prime. Although \cref{no-k+1-in-line} only gives no-$(k+1)$-in-line sets for large $k$, we remark that our constructions are purely probabilistic and do not exhibit any algebraic structure.

\subsection{Proof ideas}

Let $\mc{L}$ be the set of lines intersecting the $n \times n$ grid in at least two points. For each line $L \in \mc{L}$, define its \emph{weight} as 
\[
w(L) := \frac{|[n]^2 \cap L|}{n} \in [2/n, 1].
\]

The framework of Kov\'acs, Nagy, and Szab\'o \cite{KNS-25} allows us to reduce \cref{no-k+1-in-line} to \cref{relaxed-version} below, which permits a $0.01k$ excess over the expected number of points on the lines that are not horizontal or vertical. Therefore, the main content of this paper is the proof of \cref{relaxed-version} (though, for the reader's convenience, we also present a deduction of \cref{no-k+1-in-line} in \cref{sec:from-relaxed-to-main}).

\begin{theorem} \label{relaxed-version}
    Let $n, k$ be integers such that $10^{36} \le k \le 0.9 \cdot n$. Then there exists a subset $S \subseteq [n]^2$ that contains exactly $k$ points on each horizontal and each vertical line, and at most $k\cdot(w(L) + 0.01)$ points on every other line $L \in \mc{L}$. 
\end{theorem}

Kov\'acs, Nagy, and Szab\'o \cite{KNS-25} were able to prove a result along the lines of \cref{relaxed-version} under the much stronger assumption that $k$ is at least about $\sqrt{n\log n}$. They did this by choosing the desired point set \emph{randomly}, using concentration inequalities to show that the number of points on each line is unlikely to deviate much from its expectation.
This works when $k$ grows sufficiently fast as a function of $n$ but breaks down for constant $k$: the main issue is that a union bound over all lines becomes too inefficient.

To overcome this issue, we build on ideas from the work of Jain and Pham \cite{jain-pham-24} on optimal thresholds for Latin squares in random hypergraphs. Namely, we design an iterative subsampling procedure that uses the Lov\'asz Local Lemma at each step, maintaining control over the number of points on each line (roughly speaking, we maintain tight control over the number of points on lines whose weight exceeds some threshold, gradually raising this threshold as the procedure continues and the light lines become increasingly unimportant).
Moreover, at each step of the procedure, instead of constructing a single configuration of points, we construct a ``spread'' probability distribution over suitable configurations. We then use this spread property to show that a sample from the final distribution obtained from our procedure is likely to satisfy a certain Hall-type condition, which provides a subset containing exactly $k$ points on each horizontal and each vertical line.

\subsection{Higher dimensions} A further generalisation of this problem, introduced by Brass and Knauer \cite{brass-knauer-03}, is to ask for the maximum size of a subset of the $d$-dimensional grid $[n]^d$ that contains at most $k$ points in each affine subspace of dimension $t$ (where $1 \le t \le d-1$). Denoting this maximum by $f_{k, d, t}(n)$, the trivial upper bound is $f_{k, d, t}(n) \le k n^{d-t}$. 

Brass and Knauer \cite[Lemma 8]{brass-knauer-03} proved that $f_{k, d, t}(n) = \Omega(n^{d-t-(d(t+1))/k})$, which was slightly sharpened by Lefmann \cite[Lemma 2]{lefmann-12} to $\Omega(n^{d-t-(t(d+1))/k})$.
Sudakov and Tomon \cite[Theorem 1.4]{sudakov-tomon-24} observed that $f_{k, d, t}(n) = \Omega_d(n^{d-t})$ when $k$ is sufficiently large in terms of $d$, using a connection to optimal subspace evasive sets over finite fields; different constructions of such sets were given by Dvir and Lovett~\cite{dvir-lovett}, Conlon~\cite{conlon}, and Sudakov and Tomon~\cite{sudakov-tomon-24}. Later, Ghosal, Goenka, and Keevash~\cite[Theorem~1.1]{ghosal-goenka-keevash-25} extended the bound $f_{k, d, t}(n) = \Omega_d(n^{d-t})$ to all $k \ge d+1$. 

However, in general, the known lower bounds differ from the trivial upper bound by at least a multiplicative constant factor. We close this gap asymptotically as $k \to \infty$ by showing that in this regime $f_{k, d, t}(n) = (1-o_d(1)) k n^{d-t}$.

\begin{theorem} \label{higher-dimensions}
    For an integer $d$ and $\eps > 0$, there exists $K_1 = K_1(\eps, d)$ such that the following holds. For all integers $n, k, t$ such that $K_1 \le k \le n^t$ and $1 \le t \le d-1$, there exists a subset $S \subseteq [n]^d$ that contains at most $k$ points in each $t$-dimensional affine subspace of $\RR^d$, and contains at least $(1-\eps)k$ points in each axis-aligned $t$-dimensional affine subspace of $\RR^d$ that intersects $[n]^d$. As a consequence,
    \[
    f_{k, d, t}(n) \ge (1-\eps) k n^{d-t}.
    \]
\end{theorem}

The proof of \cref{higher-dimensions} follows the same overall scheme as the proof of \cref{relaxed-version}, but is significantly simpler (because it does not aim for an exact bound). So, the rest of the paper is organised as follows. After reviewing several known results in \cref{sec:preliminaries}, we present the short proof of \cref{higher-dimensions} in \cref{sec:higher-dimensions}. Then, we prove \cref{relaxed-version} in \cref{sec:relaxed-version}, use it to derive \cref{no-k+1-in-line} in \cref{sec:from-relaxed-to-main}, and discuss related questions and further directions in \cref{sec:concluding-remarks}.

\subsection*{Acknowledgements} We would like to thank Huy Pham for insightful discussions about the proof approach, and we would like to thank the anonymous referees for their careful reading and helpful comments.

\section{Preliminaries}\label{sec:preliminaries}

\subsection*{Notation} For a positive integer $n$, we write $[n]=\{1, \ldots, n\}$. For non-negative reals $a, b, c$ we use the shorthand $(a \pm b) c$ to denote the interval $[(a-b) c, (a+b) c]$. For an event $\mc{E}$ in some probability space, we denote its complement by $\overline{\mc{E}}$. For functions $f=f(n)$ and $g=g(n)$, we write $f=O(g)$ to mean that there is a constant $C$ such that $|f| \le C|g|$, and $f=\Omega(g)$ to mean that there is a constant $c>0$ such that $f(n) \ge c|g(n)|$ for sufficiently large $n$. Subscripts on asymptotic notation indicate quantities that should be treated as constants. For a (hyper)graph $G$, we write $V(G)$ and $E(G)$ for its vertex and edge sets, respectively. For a graph $G$ and subsets $A, B \subseteq V(G)$, we write $E_G(A, B)$ for the set of edges between $A$ and $B$ in $G$, and $e_G(A, B) = |E_G(A, B)|$. For a vertex $v$ of a graph $G$, we write $\deg_G(v)$ for its degree in $G$.

We use the following standard version of the Chernoff bound.

\begin{proposition} \label{Chernoff}
    Let $X_1, \ldots, X_n$ be independent $\Ber(p)$ random variables. Then, for every $\delta \in (0, 1)$, 
    \[
    \prob{X_1 + \ldots + X_n \notin (1 \pm \delta) pn} \le 2 \exp\left(-\frac{\delta^2 pn}{3}\right).
    \] 
\end{proposition}

The following version of the Lov\'asz Local Lemma (also employed in \cite{HSS-11,jain-pham-24}) provides control over an ensemble of independent random variables, conditioned on the non-occurrence of a family of ``bad events''. Its proof follows from the standard inductive proof of the Lov\'asz Local Lemma (see \cite[Chapter 5.1]{alon-spencer-16}).

\begin{proposition}[{\cite[Proposition 6]{jain-pham-24}}] \label{local-lemma}
    Given independent random variables $\{\xi_i\}_{i \in I}$ and events $\{\mc{E}_j\}_{j \in J}$, where each event $\mc{E}_j$ depends on a subset $X_j \subseteq I$ of variables. Assume each event has probability at most $q$, and that each set $X_j$ intersects at most $\Delta$ other sets from $\{X_{j'}\}_{j' \in J}$. If $4 q \Delta \le 1$, then 
    \begin{enumerate}
        \item[1.] $\mathbb{P}\big[\bigcap_{j \in J} \overline{\mc{E}_j}\,\big] > 0$, and
        \item[2.] for every event $\mc{E}$ depending on a subset of variables $X \subseteq I$ which intersects at most $M$ of the sets $\{X_j\}_{j \in J}$,
        \[
        \mathbb{P}\Big[\mc{E}\,\Big|\, \bigcap_{j \in J} \overline{\mc{E}_j}\Big] \le \mathbb{P}[\mc{E}] \cdot \exp(6qM).
        \]
    \end{enumerate}
\end{proposition}

Also, we recall the classical Ore--Ryser criterion for the existence of a $k$-regular spanning subgraph in a bipartite graph (see e.g.\ \cite[\S7, Ex.\ 16]{lovasz-book}), which also follows from the max-flow--min-cut theorem.

\begin{proposition} \label{max-flow-min-cut}
    A bipartite graph $G$ with parts $A_0$ and $B_0$ of equal size has a $k$-regular spanning subgraph if and only if for all subsets $A \subseteq A_0$ and $B \subseteq B_0$, we have
    \[
    e_G(A, B_0 \setminus B) \ge k(|A| - |B|).
    \]
\end{proposition}

Finally, we need the following facts about the number of integer points on lines in the plane (\cref{polynomial-tails-lines}), and, more generally, on $t$-dimensional affine subspaces in $\RR^d$ (\cref{polynomial-tails-subspaces}).

\begin{proposition} \label{polynomial-tails-lines}
    Let $\mc{L}$ be the set of lines in $\RR^2$ intersecting the grid $[n]^2$ in at least two points. Then, for every point $(x_1, x_2) \in [n]^2$ and $\alpha \in (0, 1]$, 
    \[
    \big|\big\{L \in \mc{L} : (x_1, x_2) \in L, \; |[n]^2 \cap L| \ge \alpha n \big\}\big| \le 18\alpha^{-2}.
    \]
\end{proposition}
\begin{proof}
    If $\alpha \le 2/n$, then the total number of lines in $\mc{L}$ containing $(x_1, x_2)$ is at most $n^2 \le 4\alpha^{-2}$. Otherwise, let $\ell := \lceil 2\alpha^{-1} \rceil \le n$. If $(x_1, x_2)$ is the only integer point within the rectangle $[x_1, x_1+\ell-1] \times [x_2-(\ell-1), x_2+(\ell-1)]$ on some line $L \in \mc{L}$, then
    \[
    w(L) = \frac{|[n]^2 \cap L|}{n} < \frac{n/\ell + 1}{n} \le \frac{2n/\ell}{n} \le \alpha.
    \]
    Therefore, since every two points determine a unique line, the desired number of lines is bounded by the number of integer points in this rectangle, which is at most $2\ell^2 \le 2(2\alpha^{-1}+1)^2 \le 18\alpha^{-2}$.
\end{proof}

\begin{proposition} \label{polynomial-tails-subspaces}
    Let $n, d, t$ be integers such that $1 \le t \le d-1$, and let $\mc{F}_{d, t}$ denote the set of $t$-dimensional affine subspaces of $\RR^d$. Then, for every point $x \in [n]^d$ and $\alpha \in (0, 1]$, 
    \begin{equation} \label{eq:H-count}
    \big|\big\{ F \subseteq [n]^d : x \in F,\; |F| \ge \alpha n^t, \; F = F' \cap [n]^d \text{ for some } F' \in \mc{F}_{d, t} \big\}\big| = O_d(\alpha^{-d}).
    \end{equation}
\end{proposition}
\begin{proof}
    Consider a set $F = F' \cap [n]^d$ satisfying the conditions in \cref{eq:H-count}, and let $\Gamma \subseteq \ZZ^d$ be the lattice generated by $F - x$.
    Note that $F \subseteq (\Gamma + x) \cap [n]^d \subseteq F' \cap [n]^d = F$, and hence $F = (\Gamma + x) \cap [n]^d$. In particular, $F$ is uniquely determined by $\Gamma$, and
    \begin{equation} \label{eq:Gamma-lower-bound}   
    |\Gamma \cap [-n, n]^d| \ge |F| \ge \alpha n^t.
    \end{equation}

    Denote the rank of $\Gamma$ by $t' \le t$. Let $v_1, \ldots, v_{t'} \in F - x \subseteq [-n, n]^d$ be some linearly independent vectors, let $\Gamma' \subseteq \Gamma$ be the lattice generated by them, and let $V = \operatorname{span}_{\RR}(\Gamma') = \operatorname{span}_{\RR}(\Gamma)$.
    Note that $D = \big\{\sum_{i = 1}^{t'} c_i v_i : c_1, \ldots, c_{t'} \in [0, 1)\big\} \subseteq [-t'n, t'n]^d \cap V$ is a fundamental domain for $\Gamma'$ (and thus its determinant $\det \Gamma'$ is equal to the $t'$-dimensional volume of $D$). Then, for every $z \in V$ and two distinct points $y_1, y_2 \in (\Gamma'+z) \cap [-n, n]^d$, the translates $y_1 + D$ and $y_2 + D$ are disjoint and contained in $[-(t'+1)n, (t'+1)n]^d \cap V$. Hence, $\bigsqcup_{y \in (\Gamma'+z) \cap [-n, n]^d} (y + D)$ is contained in the $t'$-dimensional Euclidean ball in $V$ of radius $\sqrt{d} (t'+1)n$ centred at the origin. Therefore, as $t' \le t \le d$,
    \[
    |(\Gamma'+z) \cap [-n, n]^d| \le \frac{\operatorname{vol}_{t'}(B(0, \sqrt{d}(t'+1)n))}{\operatorname{vol}_{t'}(D)} = O_d\left(\frac{n^{t}}{\det \Gamma'}\right).
    \]
    Summing over all $[\Gamma : \Gamma']$ cosets of $\Gamma'$ in $\Gamma$ and using that $\det \Gamma' = [\Gamma : \Gamma'] \cdot \det \Gamma$, we conclude that $|\Gamma \cap [-n, n]^d| = O_d(n^t/\det \Gamma)$.
    By \cref{eq:Gamma-lower-bound}, this implies that $\det \Gamma = O_d(\alpha^{-1})$. The desired statement now follows from a result of Schmidt \cite[Theorem 2]{schmidt-68}, which states that for every $M \ge 0$, the number of lattices $\Gamma \subseteq \ZZ^d$ with $\det \Gamma \le M$ is $O_d(M^d)$.
\end{proof}

\section{Proof of \cref{higher-dimensions}} \label{sec:higher-dimensions}

It might be possible to prove the bound $f_{k, d, t}(n) \ge (1-\eps)kn^{d-t}$ (for $k$ sufficiently large in terms of $\eps, d$) using some type of probabilistic deletion method, in the spirit of \cite{ghosal-goenka-keevash-25}. However, we take a different approach based on repeated applications of the Lov\'asz Local Lemma. This gives the stronger ``almost regularity'' property in the statement of \cref{higher-dimensions} (that every axis-aligned subspace of dimension $t$ has about $k$ points), and serves as a warm-up for the proof of \cref{relaxed-version}.

Our proof of \cref{higher-dimensions} does not rely heavily on the specific structure of affine subspaces intersecting the integer grid. The following definition captures the property of main importance to us.

\begin{definition}[polynomial tails] \label{def:polynomial-tails}
    We say that a hypergraph $\mc{H}$ (of mixed uniformity) \emph{has polynomial tails (with parameters $N, C \in \NN$)}, if every edge of $\mc{H}$ contains at most $N$ vertices, and for every vertex $v \in V(\mc{H})$ and $\alpha \in (0, 1]$, 
    \[
    |\{F \in E(\mc{H}) : v \in F, \; w(F) \ge \alpha\}| \le C\alpha^{-C},
    \] 
    where $w(F) := |F|/N$ is the \emph{weight} of an edge $F$.
\end{definition}

\begin{definition} \label{def:H_{n,d,t}}
    Let $n, d, t$ be integers such that $1 \le t \le d-1$, and let $\mc{F}_{d, t}$ denote the set of $t$-dimensional affine subspaces of $\RR^d$.
    Define $\mc{H}_{n, d, t}$ as the hypergraph on the vertex set $[n]^d$ with 
    \[
    E(\mc{H}_{n, d, t}) = \{ F \subseteq [n]^d : F = F' \cap [n]^d \text{ for some } F' \in \mc{F}_{d, t}\}.
    \]
    Since every $t$-dimensional affine subspace of $\RR^d$ intersects $[n]^d$ in at most $n^t$ points, \cref{polynomial-tails-subspaces} immediately implies that $\mc{H}_{n, d, t}$ has polynomial tails with $N = n^t$ and some $C = C(d)$.
\end{definition}

\begin{definition}[good set] \label{def:good} 
    Let $\mc{H}$ be a hypergraph having polynomial tails (with parameters $N$ and $C$). For a real number $m \in [1, N]$, we say that a set $S \subseteq V(\mc{H})$ is \emph{$m$-good} if it satisfies the following conditions:
    \begin{enumerate}
        \item (concentration for ``heavy'' edges) For every edge $F \in E(\mc{H})$ with $w(F) \in (m^{-2/3}, 1]$, we have
        \[
        |S \cap F| \in (1 \pm m^{-1/12}) m \cdot w(F).
        \]
        \item (upper bound for ``medium'' edges) For every edge $F \in E(\mc{H})$ with $w(F) \in (m^{-2}, m^{-2/3}]$, we have
        \[
        |S \cap F| \le (1 + m^{-1/12}) m^{1/3}.
        \]
        \item (very few vertices in ``light'' edges) For every edge $F \in E(\mc{H})$ with $w(F) \le m^{-2}$, we have
        \[
        |S \cap F| \le 6C + 3.
        \]
    \end{enumerate}
\end{definition}

\begin{theorem} \label{good-sets-in-hypergraphs}
    Let $\mc{H}$ be a hypergraph having polynomial tails (with parameters $N$ and $C$). Then there is $K = K(C)$ such that for every $m$ with $K \le m \le N$, there exists an $m$-good subset of $V(\mc{H})$.
\end{theorem}

\begin{proof}[\textbf{Proof of \cref{higher-dimensions} assuming \cref{good-sets-in-hypergraphs}}]
    Recall from \cref{def:H_{n,d,t}} that the hypergraph $\mc{H}_{n, d, t}$ has polynomial tails with parameters $N = n^t$ and some $C = C(d)$. We may assume that $\eps < 1/2$, and that $k$ is sufficiently large in terms of $\eps, d, C$. Then, by \cref{good-sets-in-hypergraphs}, there exists an $(1-\eps/2)k$-good subset $S$ of $[n]^d$.

    We claim that $S$ has the desired properties. Indeed, if $F'$ is a $t$-dimensional affine subspace, then $F = F' \cap [n]^d$ is an edge of $\mc{H}_{n, d, t}$. Thus, by conditions (1), (2), (3) from \cref{def:good},
    \[
    |S \cap F| \le \max\left((1 + ((1-\eps/2)k)^{-1/12})(1-\eps/2)k, 6C + 3\right) \le k.
    \]
    For the lower bound, we note that if $F'$ is an axis-aligned $t$-dimensional affine subspace intersecting $[n]^d$, then the corresponding edge $F = F' \cap [n]^d$ of $\mc{H}_{n, d, t}$ has weight $1$. Therefore, by condition (1) from \cref{def:good},
    \[
    |S \cap F| \ge (1 - ((1-\eps/2)k)^{-1/12})(1-\eps/2)k \ge (1-\eps) k.
    \]
    This completes the proof.
\end{proof}

To prove \cref{good-sets-in-hypergraphs}, we consider a sequence of real numbers $m_{(1)} < m_{(2)} < \ldots < m_{(r)}$ defined as
\begin{equation} \label{eq:m_i-higher-dimensions}
m_{(1)} = m, \quad m_{(i+1)} = m_{(i)}^3, \quad N^{1/3} < m_{(r)} \le N.
\end{equation}
Then, we prove the existence of an $m_{(i)}$-good set for each $i = r, r-1, \ldots, 1$ via a downward induction, applying the Lov\'asz Local Lemma (\cref{local-lemma}(1)) at each step. This way, \cref{good-sets-in-hypergraphs} easily follows from \cref{one-iteration-hypergraphs} below.

\begin{lemma} \label{one-iteration-hypergraphs}
    Let $\mc{H}$ be a hypergraph having polynomial tails (with parameters $N$ and $C$). Let $m_0, m$ be real numbers such that $K \le m \le m_0 \le N$ for some $K = K(C)$. Suppose that either 
    \begin{enumerate}
        \item [(a)] (induction step) $m_0 = m^3$ and $S_0$ is an $m_0$-good subset of $V(\mc{H})$, or
        \item [(b)] (base case) $m_0 = N \le m^3$ and $S_0 = V(\mc{H})$ (which is trivially $m_0$-good).
    \end{enumerate}
    Then there exists an $m$-good set $S \subseteq S_0$.    
\end{lemma}

\begin{proof}
    Let $\delta := m^{-1/12}$. We will assume that $m$ is sufficiently large in terms of $C$ (and, in particular, that $\delta \le 0.01$).
    We claim that for an edge $F \in E(\mc{H})$,
    \begin{equation} \label{eq:heavy-medium-edges}
    |S_0 \cap F| \in (1 \pm \delta^3) m_0 \cdot w(F) \quad \text{ if }  w(F) \in (m^{-2}, 1];
    \end{equation}
    \begin{equation} \label{eq:light-edges}
    |S_0 \cap F| \le (1 + \delta^3) m_0 \cdot m^{-2} \quad \text{ if } w(F) \in (m_0^{-2}, m^{-2}].
    \end{equation}
    Indeed, in case (a), since $m_0 = m^3$ and $\delta^3 = m_0^{-1/12}$, this follows from conditions (1) and (2) from the definition of an $m_0$-good set. On the other hand, in case (b) we have $m_0 = N$ and $S_0 = V(\mc{H})$, and hence $|S_0 \cap F| = m_0 \cdot w(F)$ for every edge $F \in E(\mc{H})$ by the definition of weight.

    Let $p := m/m_0$, and let $S$ be a random subset of $S_0$ obtained by including each vertex independently with probability $p$. We will show that $S$ is $m$-good with positive probability, through an application of the Lov\'asz Local Lemma (\cref{local-lemma}(1)). First, we identify the bad events and estimate their probabilities.

    \textbf{Heavy edges.} For an edge $F \in E(\mc{H})$ with $w(F) \in (m^{-2/3}, 1]$, let $\mc{E}_1(F)$ denote the bad event that $S$ does not satisfy condition (1) for this edge $F$: that is, $|S \cap F| \notin (1 \pm \delta) m \cdot w(F)$. By \cref{eq:heavy-medium-edges}, for such edges $F$ we have $|S_0 \cap F| \in (1 \pm \delta^3) m_0 \cdot w(F)$. Thus, if $\mc{E}_1(F)$ occurs, then
    \[
    |S \cap F| \notin (1 \pm \delta/2) p |S_0 \cap F|.
    \]
    Therefore, by the Chernoff bound (\cref{Chernoff}),
    \[
    \prob{\mc{E}_1(F)} \le 2\exp\left(-\frac{(\delta/2)^2 p |S_0 \cap F|}{3}\right) \le 2\exp\left(-\frac{(\delta/2)^2 (1-\delta^3) m \cdot w(F)}{3}\right) \le 2\exp\left(-m^{1/6}/13\right).
    \]

    \textbf{Medium edges.} For an edge $F \in E(\mc{H})$ with $w(F) \in (m^{-2}, m^{-2/3}]$, let $\mc{E}_2(F)$ denote the bad event that $S$ does not satisfy condition (2) for this edge $F$. By \cref{eq:heavy-medium-edges}, for such edges $F$ we have
    \[
    |S_0 \cap F| \le m' := \lfloor (1 + \delta^3) m_0 \cdot m^{-2/3}\rfloor.
    \]
    Let $X$ be a sum of $m'$ independent $\Ber(p)$ random variables. Then, by the Chernoff bound (\cref{Chernoff}),
    \begin{align*}
    \prob{\mc{E}_2(F)} &= \prob{|S \cap F| > (1+\delta) m^{1/3}} \le \prob{X > (1+\delta) m^{1/3}} \le \prob{X \notin (1 \pm \delta/2) p m'} \\
    &\le 2\exp\left(-\frac{(\delta/2)^2 p m'}{3}\right) \le 2\exp\left(-\frac{(\delta/2)^2 (m^{1/3}-1)}{3}\right) \le 2\exp\left(-m^{1/6}/13\right).
    \end{align*}

    \textbf{Light edges.} Note that for edges $F$ with $w(F) \le m_0^{-2}$, condition (3) for $S$ is implied by condition (3) for $S_0$, so it suffices to consider only edges with $w(F) \in (m_0^{-2}, m^{-2}]$. Let $\mc{E}_3(F)$ denote the bad event that $S$ does not satisfy condition (3) for such an edge $F$. By \cref{eq:light-edges}, we have $|S_0 \cap F| \le (1 + \delta^3) m_0 m^{-2}$.
    Therefore, denoting $C_1 := 6C + 4$,
    \[
    \prob{\mc{E}_3(F)} = \prob{|S \cap F| \ge C_1} \le \binom{|S_0 \cap F|}{C_1} \cdot p^{C_1} \le \left((1+\delta^3) m_0 m^{-2} \cdot m/m_0\right)^{C_1} \le \left(2/m\right)^{C_1}.
    \]
    Since $m$ is sufficiently large in terms of $C$, the maximum probability $q$ of a bad event then satisfies
    \[
    q \le \max\left(2\exp(-m^{1/6} / 13), \; (2/m)^{6C + 4}\right) = (2/m)^{6C + 4}.
    \]

    \textbf{Degree of the dependency graph.} So, we have exactly one bad event ($\mc{E}_1$, $\mc{E}_2$, or $\mc{E}_3$) for each edge $F \in E(\mc{H})$ with $w(F) \in (m_0^{-2}, 1]$. Note that each of these events depends only on the Bernoulli random variables associated with the underlying vertices of $S_0 \cap F$, and the number of such vertices satisfies $|S_0 \cap F| \le (1+\delta^3) m_0$ by \cref{eq:heavy-medium-edges} and \cref{eq:light-edges}.

    Recall that $\mc{H}$ has polynomial tails (with parameters $N$ and $C$). Therefore, for every vertex $v \in V(\mc{H})$ we can bound the number of edges with large weight containing $v$ as follows:
    \[
    \left|\{F \in E(\mc{H}) : v \in F, \; w(F) \ge m_0^{-2}\}\right| \le C m_0^{2C}.
    \]
    As $m_0 \le m^3$, the maximum degree $\Delta$ of the dependency graph of our bad events then satisfies
    \[
    \Delta \le (1 + \delta^3) m_0 \cdot C m_0^{2C} \le 2C m^{6C+3}, \quad \text{ and thus } \quad 4q\Delta \le C 2^{6C+7} / m.
    \]
    Since $m$ is sufficiently large in terms of $C$, \cref{local-lemma}(1) completes the proof.
\end{proof}

\begin{proof}[\textbf{Proof of \cref{good-sets-in-hypergraphs}}]
    Consider a sequence of real numbers $m = m_{(1)} < m_{(2)} < \ldots < m_{(r)}$ satisfying \cref{eq:m_i-higher-dimensions}. By \cref{one-iteration-hypergraphs}(b), there exists an $m_{(r)}$-good subset of $V(\mc{H})$. Iteratively applying \cref{one-iteration-hypergraphs}(a), we conclude that there exists an $m_{(1)}$-good subset of $V(\mc{H})$, as required. 
\end{proof}

\section{Proof of \cref{relaxed-version}} \label{sec:relaxed-version}

In this section, our goal is to prove \cref{relaxed-version}. Here we do not use the language of hypergraphs with polynomial tails: though some parts of our argument can be translated to that setting, other parts substantially rely on the fact that the grid $[n]^2$ can be interpreted as the complete bipartite graph $K_{n, n}$. \cref{def:nice} below provides a suitable analogue of good sets from \cref{def:good}.

Recall that $\mc{L}$ denotes the set of lines in $\RR^2$ intersecting the grid $[n]^2$ in at least two points, and that for every line $L \in \mc{L}$, 
\[
w(L) = \frac{|[n]^2 \cap L|}{n}.
\]

\begin{definition}[nice set] \label{def:nice}
    For a real number $m \in [1, n]$, we say that a set $S \subseteq [n]^2$ is \emph{$m$-nice} if it satisfies the following conditions:
    \begin{enumerate}
        \item (concentration for ``heavy'' lines) For every line $L \in \mc{L}$ with $w(L) \in (m^{-2/3}, 1]$, we have 
        \[
        |S \cap L| \in (1 \pm m^{-1/12}) m \cdot w(L).
        \]
        \item (upper bound for ``medium'' lines) For every line $L \in \mc{L}$ with $w(L) \in (m^{-2}, m^{-2/3}]$, we have 
        \[
        |S \cap L| \le (1 + m^{-1/12})m^{1/3}.
        \]
        \item (very few points on ``light'' lines) For every line $L \in \mc{L}$ with $w(L) \le m^{-2}$, we have 
        \[
        |S \cap L| \le C := 14.
        \]
        \item (quasirandomness) For every pair of subsets $I, J \subseteq [n]$ of size at least $n/10$, we have
        \[
        |S \cap (I \times J)| \in (1 \pm m^{-1/12}) |I| |J| \cdot m/n.
        \]
    \end{enumerate}
\end{definition}

For the rest of this section, we fix $K := 10^{36}$ and $\eps := 0.005$.

Conditions (1), (2), (3) from \cref{def:nice} imply that, for $k \ge K$, a $(1+\eps)k$-nice set contains at most $k \cdot (w(L) + 0.01)$ points on each line $L \in \mc{L}$. To prove \cref{relaxed-version}, we will take such a set and ``regularise'' it: that is, we will find a subset containing exactly $k$ points on each horizontal and each vertical line. 

This regularisation step is the reason we need the quasirandomness condition in \cref{def:nice}. However, this quasirandomness condition is not enough on its own, and we need a further idea of Jain and Pham \cite{jain-pham-24}. Indeed, instead of constructing a single nice set, we construct a sufficiently ``spread'' probability distribution supported on nice sets, and prove that a sample from this distribution can be regularised with positive probability.

Formally, we say that a probability distribution on subsets of $[n]^2$ is \emph{$p$-spread} if, for every subset $T \subseteq [n]^2$ and a sample $S$ from this distribution,
\[
\prob{T \subseteq S} \le p^{|T|}.
\]

Our proof of \cref{relaxed-version} follows the same general strategy as the proof of \cref{higher-dimensions}. Namely, given $n \ge k \ge K$, we consider a sequence of real numbers $m_{(1)} < m_{(2)} < \ldots < m_{(r)}$ such that
\begin{equation} \label{eq:m_i-plane}
m_{(1)} = (1 + \eps)k, \quad m_{(i+1)} = m_{(i)}^3, \quad n^{1/3} < m_{(r)} \le n,
\end{equation}
and find $(m_{(i)}+1)/n$-spread distributions supported on $m_{(i)}$-nice sets for $i = r, r-1, \ldots, 1$ via a downward induction.
We handle the base case and the induction step of this procedure in \cref{induction-step}(b) and \cref{induction-step}(a), respectively. Then, in \cref{regularisation}, we perform the final regularisation step: specifically, we show that a sample from the $(m_{(1)}+1)/n$-spread distribution supported on $m_{(1)}$-nice sets has a subset containing exactly $k$ points on each horizontal and each vertical line with positive probability.

\begin{lemma} \label{induction-step}
    Let $n$ be an integer and $m_0, m$ be real numbers such that $K \le m \le m_0 \le n$. Suppose that either 
    \begin{enumerate}
        \item [(a)] $m_0 = m^3$ and there exists an $(m_0+1)/n$-spread distribution supported on $m_0$-nice sets, or
        \item [(b)] $m_0 = n \le m^3$ (in this case, we have the trivial  $1$-spread distribution supported on the $m_0$-nice set $[n]^2$).
    \end{enumerate}
    Then there exists an $(m+1)/n$-spread distribution supported on $m$-nice sets.
\end{lemma}
\begin{proof}
    Let $\delta := m^{-1/12} \le 0.001$. Let $S_0$ be an $m_0$-nice set sampled from the ``original distribution'' (provided by (a) or (b)). We first claim that for each line $L \in \mc{L}$,
    \begin{equation} \label{eq:heavy-medium-lines}
    |S_0 \cap L| \in (1 \pm \delta^3) m_0 \cdot w(L) \quad \text{ if }  w(L) \in (m^{-2}, 1];
    \end{equation}
    \begin{equation} \label{eq:light-lines}
    |S_0 \cap L| \le (1 + \delta^3) m_0 \cdot m^{-2} \quad \text{ if }  w(L) \in (m_0^{-2}, m^{-2}].
    \end{equation}
    Indeed, in case (a), since $m_0 = m^3$ and $\delta^3 = m_0^{-1/12}$, this follows from conditions (1) and (2) from the definition of an $m_0$-nice set. On the other hand, in case (b), we have $m_0 = n$ and $S_0 = [n]^2$, and hence $|S_0 \cap L| = m_0 \cdot w(L)$ for every line $L$ by the definition of weight.

    Let $p := m/m_0$, and let $S$ be a random subset of $S_0$ obtained by including each point with probability $p$ (independently of each other and of the randomness of $S_0$). For most of the proof, we will not actually need the randomness of $S_0$, so we fix an arbitrary outcome of $S_0$, and implicitly work in the corresponding conditional probability space.
    
    To apply the Lov\'asz Local Lemma (\cref{local-lemma}), we first identify the bad events and estimate their probabilities (these calculations are almost identical to those in the proof of \cref{one-iteration-hypergraphs}). 
    
    \textbf{Heavy lines.} For a line $L \in \mc{L}$ with $w(L) \in (m^{-2/3}, 1]$, let $\mc{E}_1(L)$ denote the bad event that $S$ does not satisfy condition (1) from \cref{def:nice} for this line $L$: that is, $|S \cap L| \notin (1 \pm \delta) m \cdot w(L)$. By \cref{eq:heavy-medium-lines}, for such lines $L$ we have
    $|S_0 \cap L| \in (1 \pm \delta^3) m_0 \cdot w(L)$. Thus, if $\mc{E}_1(L)$ occurs, then 
    \[
    |S \cap L| \notin (1 \pm \delta/2) p |S_0 \cap L|.
    \]
    Therefore, by the Chernoff bound (\cref{Chernoff}),
    \[
    \prob{\mc{E}_1(L)} \le 2\exp\left(-\frac{(\delta/2)^2 p |S_0 \cap L|}{3}\right) \le 2\exp\left(-\frac{(\delta/2)^2 (1-\delta^3) m \cdot w(L)}{3}\right) \le 2\exp\left(-m^{1/6}/13\right).
    \]

    \textbf{Medium lines.} For a line $L \in \mc{L}$ with $w(L) \in (m^{-2}, m^{-2/3}]$, let $\mc{E}_2(L)$ denote the bad event that $S$ does not satisfy condition (2) for this line $L$. By \cref{eq:heavy-medium-lines}, for such lines $L$ we have
    \[
    |S_0 \cap L| \le m' := \lfloor (1+\delta^3) m_0 \cdot m^{-2/3} \rfloor.
    \]
    Let $X$ be a sum of $m'$ independent $\Ber(p)$ random variables. Then, by the Chernoff bound (\cref{Chernoff}),
    \begin{align*}
    \prob{\mc{E}_2(L)} &= \prob{|S \cap L| > (1+\delta) m^{1/3}} \le \prob{X > (1+\delta) m^{1/3}} \le \prob{X \notin (1 \pm \delta/2) p m'} \\
    &\le 2\exp\left(-\frac{(\delta/2)^2 p m'}{3}\right) \le 2\exp\left(-\frac{(\delta/2)^2 (m^{1/3}-1)}{3}\right) \le 2\exp\left(-m^{1/6}/13\right).
    \end{align*}

    \textbf{Light lines.} Note that for lines with $w(L) \le m_0^{-2}$, condition (3) for $S$ is implied by condition (3) for $S_0$, so it suffices to consider only lines with $w(L) \in (m_0^{-2}, m^{-2}]$.
    Let $\mc{E}_3(L)$ denote the bad event that $S$ does not satisfy condition (3) for such a line $L$. By \cref{eq:light-lines}, we have $|S_0 \cap L| \le (1 + \delta^3) m_0 m^{-2}$. Therefore, recalling that $C = 14$,
    \[
    \prob{\mc{E}_3(L)} = \prob{|S \cap L| \ge C+1} \le \binom{|S_0 \cap L|}{C+1} \cdot p^{C+1} \le \left(\frac{e \cdot (1+\delta^3)m_0 m^{-2}}{C+1} \cdot p\right)^{C+1} \le \left(\frac{1}{5m}\right)^{15}.
    \]
    (Here we used the standard estimate $\binom{n}{k} \le (en/k)^k$.)

    Since $m \ge K = 10^{36}$, the maximum probability $q$ of a bad event then satisfies 
    \[
    q \le \max(2\exp(-m^{1/6}/13), (1/(5m))^{15}) = (1/(5m))^{15}.
    \]
    
    \textbf{Passing to the conditional distribution.} So, we have exactly one bad event for each line $L$ with weight in $(m_0^{-2}, 1]$ that intersects $S_0$. 
    Note that each of the events $\mc{E}_1(L)$, $\mc{E}_2(L)$, and $\mc{E}_3(L)$ depends only on the Bernoulli random variables associated with the underlying points of $S_0 \cap L$, and the number of such points satisfies $|S_0 \cap L| \le (1 + \delta^3)m_0$ by \cref{eq:heavy-medium-lines} and \cref{eq:light-lines}. Furthermore, by \cref{polynomial-tails-lines}, there are at most $D := 18m_0^4$ lines with weight larger than $m_0^{-2}$ passing through each point of the grid $[n]^2$. Therefore, since $m_0 \le m^3$, the maximum degree $\Delta$ of the dependency graph satisfies
    \[
    \Delta \le (1 + \delta^3)m_0 \cdot D \le 20 m_0^5 \le 20 m^{15}, \quad \text{ and thus } \quad q\Delta \le (1/(5m))^{15} \cdot 20 m^{15} < 1/4.
    \]
    Let $\mc{E}$ denote the event that none of the bad events occur. Then $\prob{\mc{E}} > 0$ by \cref{local-lemma}(1), and we can consider the conditional distribution of $S$ given $\mc E$.
    
    \textbf{Quasirandomness.} If $\mc{E}$ occurs, then $S$ clearly satisfies conditions (1), (2), and (3). Next, we check that, in the conditional probability space given $\mc E$, our random set $S$ also satisfies condition (4) with high probability (denote the event that condition (4) fails by $\mc{E}_4$).

    Since each of the $n$ horizontal lines contains at most $(1+\delta^3)m_0$ points of $S_0$, the total number of events $N$ in our application of the local lemma satisfies
    \[
    N \le n \cdot (1+\delta^3)m_0 \cdot D \le 20 m^{15} n, \quad \text{ and thus } \quad qN \le (1/(5m))^{15} \cdot 20 m^{15} n < n.
    \]
    In case (a), condition (4) for $S_0$ gives that 
    \[
    |S_0 \cap (I \times J)| \in (1 \pm \delta^3) |I| |J| \cdot m_0/n
    \] 
    for every pair of sets $I, J \subseteq [n]$ of size at least $n/10$. On the other hand, in case (b) we trivially have $|S_0 \cap (I \times J)| = |I||J|$. Either way, by the Chernoff bound (\cref{Chernoff}),
    \begin{align*}
    \mathbb{P}\big[|S \cap (I \times J)| &\notin (1 \pm \delta) |I| |J| \cdot m/n\big] \le \mathbb{P}\big[|S \cap (I \times J)| \notin (1 \pm \delta/2) p \cdot |S_0 \cap (I \times J)|\big] \\
    &\le 2\exp\left(-\frac{(\delta/2)^2 p |S_0 \cap (I \times J)|}{3}\right) \le 2\exp\left(-\frac{(\delta/2)^2 (1 - \delta^3) mn}{3 \cdot 10^2}\right) \le 2\exp\left(-\frac{m^{5/6}n}{1300}\right).
    \end{align*}
    Therefore, by the union bound and \cref{local-lemma}(2),
    \begin{equation}\begin{split} \label{eq:condition-4-fails}
    \prob{\mc{E}_4 \mid \mc{E}} &\le \prob{\mc{E}_4} \cdot \exp(6qN) \le 2^{2n} \cdot 2\exp(-m^{5/6} n/1300) \cdot \exp(6qN) \\
    &\le \exp(n \cdot (2 + 1 - (m^{5/6}/1300) + 6)) \le \exp(-n).
    \end{split}\end{equation}
    (Here we again used that $m \ge K = 10^{36}$.)
    
    \textbf{Spreadness.} So far we have been conditioning on a fixed outcome of $S_0$, and working only with the randomness of $S \subseteq S_0$. Now we incorporate the randomness of $S_0$ as well (so from now on, when we wish to consider probabilities conditional on an outcome of $S_0$, we indicate this explicitly with notation like ``$\prob{\cdot \,\middle|\, S_0}$'').
    
    Consider the distribution obtained by first sampling $S_0$ from its original distribution, and then sampling $S$ from the conditional distribution given this outcome of $S_0$ and the events $\mc{E}$ and $\overline{\mc{E}_4}$. We claim that this distribution satisfies the desired properties. It is supported on $m$-nice sets by construction, hence it only remains to check that it is $(m+1)/n$-spread.

    Consider a non-empty set $T \subseteq [n]^2$. In both cases (a) and (b), our original distribution is $((1+m^{-3})m_0/n)$-spread, and thus 
    \begin{equation} \label{eq:spreadness-original}
    \prob{T \subseteq S_0} \le \big((1+m^{-3})m_0/n\big)^{|T|} = \big((1/p) \cdot (1+m^{-3})m/n\big)^{|T|}.
    \end{equation}
    
    Condition on an outcome of $S_0$ such that $T \subseteq S_0$. Since the event ``$T \subseteq S$'' depends on at most $D \cdot |T|$ bad events (of the form $\mc{E}_1(L)$, $\mc{E}_2(L)$, or $\mc{E}_3(L)$), \cref{local-lemma}(2) implies that 
    \begin{align*}
    \prob{T \subseteq S \,\middle|\, \mc{E}, S_0} &\le \prob{T \subseteq S \,\middle|\, S_0} \cdot \exp(6q D |T|) \le (p \cdot \exp(6qD))^{|T|} \le (p \cdot (1 + m^{-3}))^{|T|}.
    \end{align*}
    Since $n \ge m$, combining this with \cref{eq:condition-4-fails} gives that
    \begin{align*}
    \prob{T \subseteq S \,\middle|\, \mc{E} \cap \overline{\mc{E}_4}, S_0} &\le \frac{\prob{T \subseteq S \,\middle|\, \mc{E}, S_0}}{1 - \prob{\mc{E}_4 \,\middle|\, \mc{E}, S_0}} \le \frac{(p \cdot (1 + m^{-3}))^{|T|}}{1 - \exp(-n)} \\
    &\le (1 + 2\exp(-m)) \cdot (p \cdot (1 + m^{-3}))^{|T|} \le (p \cdot (1 + m^{-3})^2)^{|T|}. 
    \end{align*}
    Recalling \cref{eq:spreadness-original}, we conclude that 
    \[
    \EE\left[\mathbf{1}_{T \subseteq S_0} \cdot \prob{T \subseteq S \,\middle|\, \mc{E} \cap \overline{\mc{E}_4}, S_0}\right] \le ((1 + m^{-3})^3 m/n)^{|T|} \le ((m+1)/n)^{|T|},
    \]
    which means precisely that our new distribution is $(m+1)/n$-spread.
\end{proof}

In \cref{regularisation} below, we carry out the final regularisation step. The case analysis in its proof is similar to that in \cite[Proposition 11]{jain-pham-24}.

\begin{lemma} \label{regularisation}
    Let $n, k$ be integers such that $K \le k \le 0.9 n$. Consider a $((1+\eps)k+1)/n$-spread distribution supported on $(1+\eps)k$-nice sets, and let $S \subseteq [n]^2$ be a random sample from this distribution. Then, with positive probability, there exists a subset $S' \subseteq S$ that contains exactly $k$ points on each horizontal and each vertical line.
\end{lemma}
\begin{proof}
    Our random subset $S \subseteq [n]^2$ naturally corresponds to a random bipartite graph $G$ with two parts of size $n$. Call these two parts $A_0$ and $B_0$. We need to prove that $G$ has a $k$-regular spanning subgraph. By \cref{max-flow-min-cut}, it suffices to check that, with positive probability, for every choice of subsets $A \subseteq A_0, B \subseteq B_0$, we have 
    \begin{equation} \label{eq:hall}
        e_G(A, B_0 \setminus B) \ge k(|A| - |B|).
    \end{equation}

We cover all possible choices of $A\subseteq A_0$ and $B\subseteq B_0$ by seven ``types'' (a pair of sets $(A, B)$ may have more than one type):
\begin{itemize}
    \item $(A, B)$ is of \emph{type 0} if $|B| \ge |A|$;
    \item $(A, B)$ is of \emph{type 1} if $|A| > 2|B|$;
    \item $(A, B)$ is of \emph{type 2} if $|A|\ge n/10$ and $|B|\le n/2$;
    \item $(A, B)$ is of \emph{type 3} if $|B|< |A| \le 2|B|$ and $|B|\le n/10$;
    \item $(A, B)$ is of \emph{type 1*} if $|B_0\setminus B| > 2|A_0\setminus A|$;
    \item $(A, B)$ is of \emph{type 2*} if $|B_0\setminus B|\ge n/10$ and $|A_0\setminus A|\le n/2$;
    \item $(A, B)$ is of \emph{type 3*} if $|A_0\setminus A|< |B_0\setminus B| \le 2|A_0\setminus A|$ and $|A_0\setminus A|\le n/10$.
\end{itemize}
First, we check that every pair of sets has at least one type. Indeed, if a pair of sets $(A, B)$ is not of type 0 then $|B| < |A|$, and hence $|B| \le n/2$ or $|A_0\setminus A| \le n/2$. Consider the case when $|B| \le n/2$: if $|A| \ge n/10$ then this pair is of type 2; otherwise, $|B| < |A| < n/10$, and thus this pair is of type 1 or of type 3. Similarly, in the case when $|A_0\setminus A| \le n/2$, this pair is of type 1*, 2*, or 3*.

Note that pairs of sets $(A, B)$ of type 0 trivially satisfy \cref{eq:hall}. We will show that
\begin{itemize}
    \item with probability $1$, all pairs of types 1, 1*, 2, and 2* satisfy \cref{eq:hall}; 
    \item with probability greater than $1/2$, all pairs of type 3 satisfy \cref{eq:hall};
    \item with probability greater than $1/2$, all pairs of type 3* satisfy \cref{eq:hall}. 
\end{itemize}
This will immediately imply that, with positive probability, \cref{eq:hall} holds for all choices of $A \subseteq A_0$ and $B \subseteq B_0$.

Note that \cref{eq:hall} can be equivalently written as $e_G(B_0\setminus B,A_0\setminus(A_0\setminus A))\ge k(|B_0\setminus B|-|A_0\setminus A|)$. Therefore, by symmetry, it suffices to consider types 1, 2, and 3.

Since $k \ge K = 10^{36}$ and $\eps = 0.005$, we have $\eps/4 > ((1+\eps)k)^{-1/12}$. Thus, by condition (1) from \cref{def:nice}, every vertex in $G$ has degree in $(1 \pm \eps/4)(1+\eps)k$.

\textbf{Type 1 pairs:} 
    For a pair $(A, B)$ of type 1 (with $|A| > 2|B|$), we have
    \begin{align*}
    e_G(A, B_0 \setminus B) &\ge \sum_{v \in A} \deg_G(v) - \sum_{v \in B} \deg_G(v) \ge |A| \cdot (1-\eps/4)(1+\eps)k - |B| \cdot (1+\eps/4)(1+\eps)k \\
    &\ge k(|A| - |B|) + k((3\eps/4 - \eps^2/4)|A| - (5\eps/4 + \eps^2/4)|B|) \ge k(|A| - |B|).
    \end{align*}     
\textbf{Type 2 pairs:} For a pair $(A, B)$ of type 2 (with $|A|\ge n/10$ and $|B_0 \setminus B| = n - |B| \ge n/2$), condition (4) from \cref{def:nice} implies that
    \[
    e_G(A, B_0 \setminus B) \ge (1 - \eps/4) \cdot (1+\eps) k/n \cdot |A| (n - |B|) \ge k |A| \frac{n-|B|}{n} \ge k(|A| - |B|).
    \]
\textbf{Type 3 pairs:}
    Consider a pair $(A, B)$ of type 3 (with $|B| < |A| \le 2|B|$ and $|B| \le n/10$). Note that if $e_G(A, B_0 \setminus B) < k(|A| - |B|)$, then
    \[
    e_G(A, B) > \sum_{v \in A} \deg_G(v) - k(|A| - |B|) \ge ((1-\eps/4)(1+\eps) - 1)k|A| + k|B| \ge k|B|.
    \]
    For a set of possible edges $E \subseteq A \times B$ of size $k|B|$, let $\mc{E}(A, B, E)$ be the event that $E \subseteq E_G(A, B)$.
    Since our distribution is $((1+\eps)k+1)/n$-spread, we have
    \[
    \sum_{E \subseteq A \times B,\; |E| = k|B|} \prob{\mc{E}(A, B, E)}  \le \binom{|A||B|}{k|B|}\cdot\left(\frac{(1+\eps)k+1}{n}\right)^{k|B|} \le \left(\frac{e |A|}{k} \cdot \frac{(1+\eps)k+1}{n}\right)^{k|B|} \le \left(\frac{6 |B|}{n}\right)^{k|B|}.
    \]
    For every $b \le n/10$, we sum these probabilities over all pairs of sets $(A, B)$ with $|B| = b$ and $b < |A| \le 2b$, to obtain that
    \begin{align*}
    \sum_{\substack{|B| = b \\ b < |A| \le 2b}} \sum_{\substack{E \subseteq A \times B, \\ |E| = kb}} \prob{\mc{E}(A, B, E)} &\le \binom{n}{b} \cdot b \binom{n}{2b} \cdot \left(\frac{6 b}{n}\right)^{kb} \le \left(\frac{en}{b}\right)^{3b} \cdot \left(\frac{6 b}{n}\right)^{kb} \\ 
    &\le \Big((6e)^{3} \cdot (6/10)^{k-3}\Big)^b \le (1/4)^b.
    \end{align*}
    Therefore, the probability that there is a pair of type 3 violating \cref{eq:hall} is at most
    \[
    \sum_{b=1}^{\infty} (1/4)^b < 1/2. \qedhere
    \]    
\end{proof}

\begin{proof}[\textbf{Proof of \cref{relaxed-version}}]
    Recall that $K = 10^{36}$, $\eps = 0.005$, and $K \le k \le 0.9n$. Let $(1 + \eps) k = m_{(1)} < m_{(2)} < \ldots < m_{(r)}$ be the sequence of real numbers from \cref{eq:m_i-plane}.

    Then, by \cref{induction-step}(b), there exists an $(m_{(r)}+1)/n$-spread distribution supported on $m_{(r)}$-nice sets. Iteratively applying \cref{induction-step}(a), we obtain an $(m_{(1)}+1)/n$-spread distribution supported on $m_{(1)}$-nice sets. Next, \cref{regularisation} implies that there exists an $m_{(1)}$-nice set $S$ containing a subset $S' \subseteq S$ with exactly $k$ points on each horizontal and each vertical line.

    We claim that $S'$ has the desired properties. For this, we need to check that for every line $L \in \mc{L}$, we have $|S' \cap L| \le k \cdot (w(L) + 0.01)$. Let $\delta := m_{(1)}^{-1/12} \le 0.001$. Since $S$ is $m_{(1)}$-nice, condition (1) implies that for every line $L$ with $w(L) \in (m_{(1)}^{-2/3}, 1]$,
    \[
    |S \cap L| \le (1 + \delta) m_{(1)} \cdot w(L) = (1 + \delta) (1 + \eps)k \cdot w(L) \le k \cdot (w(L) + 0.01).
    \]
    On the other hand, conditions (2) and (3) imply that for every line $L$ with $w(L) \le m_{(1)}^{-2/3}$,
    \[
    |S \cap L| \le \max\left((1 + \delta) m_{(1)}^{1/3}, 14\right) \le 0.01 \cdot k.
    \]
    This completes the proof, because $S'$ is a subset of $S$.
\end{proof}

\section{Proof of \cref{no-k+1-in-line}}
\label{sec:from-relaxed-to-main}

In this section, we deduce \cref{no-k+1-in-line} from \cref{relaxed-version} using the results of Kov\'acs, Nagy, and Szab\'o \cite{KNS-25}. The key observation behind this deduction is that the only lines with weight close to $1$ are horizontal lines, vertical lines, and the lines of slope $\pm 1$ near the main diagonals of the $n \times n$ grid.

\begin{proof}[\textbf{Proof of \cref{no-k+1-in-line}}]    
    First, we prove the statement under the additional assumption that $k \le \frac{3}{4} n$, $4 \mid n$, and $10 \mid k$. As in \cite[Definition 3.2]{KNS-25}, we partition the $n \times n$ grid into sixteen blocks of size $n/4 \times n/4$. Formally, for each $1 \le i, j \le 4$, we define 
    \[
    \mc{G}_{i, j} := \Big\{(x, y) \in [n]^2 : (i-1)\frac{n}{4} < x \le i\frac{n}{4}, \; (j-1)\frac{n}{4} < y \le j\frac{n}{4}\Big\}.
    \]
    We also set
    \[
    k_{i, j} := \begin{cases} \frac{2}{10} k, & \text{ if $i = j$ or $i + j = 5$ (i.e., if $\mc{G}_{i, j}$ intersects one of the main diagonals);} \\ \frac{3}{10} k, & \text{ otherwise,} \end{cases}
    \]
    and let $w_{i, j}(L) := |\mc{G}_{i, j} \cap L|/(n/4)$ for each line $L \in \mc{L}$.
    
    For each $i, j$, we apply \cref{relaxed-version} (with parameters $n/4$ and $k_{i,j}$; note that $10^{37} \le k \le \frac{3}{4} n$ implies that $10^{36} \le \frac{2}{10} k \le k_{i, j} \le \frac{3}{10} k \le 0.9 \cdot n/4$, and thus the necessary conditions hold) to obtain a subset $S_{i, j}\subseteq \mc{G}_{i, j}$ with exactly $k_{i,j}$ points on each horizontal and each vertical line, and at most $k_{i,j}\cdot (w_{i, j}(L)+0.01)$ points on every other line $L$. Let $S := \bigsqcup\limits_{1 \le i, j \le 4} S_{i, j}$ be the union of all these sets.

    Since each $S_{i, j}$ contains exactly $k_{i, j}$ points on each horizontal line and each vertical line, $S$ contains exactly $k$ points on each horizontal and each vertical line.
    Crucially, by \cite[Lemma 3.5]{KNS-25}, for every line $L$ which is not horizontal or vertical, we have
    \[
    \sum_{1 \le i, j \le 4} k_{i, j} w_{i, j}(L) \le \frac{4}{5}k,
    \]
    so
    \[
    |S \cap L| = \sum_{1 \le i, j \le 4} |S_{i, j} \cap L| \le \sum_{1 \le i, j \le 4} k_{i, j} (w_{i, j}(L) + 0.01) \le \frac{4}{5} k + 0.01 \sum_{1 \le i, j \le 4} k_{i, j} = 0.84 \cdot k.
    \]
    Since $k \ge 10^{37}$, this implies, in particular, that $|S \cap L| \le k - 15$. This completes the proof of \cref{no-k+1-in-line} in the case when $k \le \frac{3}{4} n$, $4 \mid n$, and $10 \mid k$.
    
    For the general case, note that if $10^{37} \le k < \frac{2}{3} n$ then $n' := 4 \cdot \lfloor n/4 \rfloor$ and $k' := 10 \cdot \lceil k/10 \rceil$ satisfy $k' \le \frac{3}{4} n'$, $4 \mid n'$, and $10 \mid k'$. Thus, the statement of the theorem for $n$ and $k$ follows from the same statement for $n'$ and $k'$ via \cite[Lemma 3.12 and Lemma 3.13]{KNS-25}, making use of the ``reserve'' of $15$ points per line. Finally, an explicit construction \cite[Proposition 3.1]{KNS-25} takes care of the case $k \ge \frac{2}{3}n$.
\end{proof}

\section{Concluding remarks} \label{sec:concluding-remarks}

\textbf{Smaller values of $k$?} We did not try to optimise the constant $K_0 = 10^{37}$ in \cref{no-k+1-in-line}, and our argument provides a lot of room for optimisation. However, making any progress on the classical no-three-in-line problem (i.e., the case $k = 2$) would likely require new ideas. Note that, even for ``light'' lines, our argument can only guarantee that they contain at most $C = 14$ points of our set (see condition (3) in \cref{def:nice}).

\textbf{Partitioning the grid into no-$(k+1)$-in-line sets.} In \cite{jain-pham-24}, Jain and Pham constructed an optimally spread probability distribution over decompositions of the edge set of $K_{n, n}$ into perfect matchings (i.e., over Latin squares of order $n$).
It is plausible that, modifying our approach in a suitable way, one also might be able to prove that for a sufficiently large $k$ and every $n$ divisible by $k$, the grid $[n]^2$ can be partitioned into $n/k$ sets of size $kn$, each containing no $k+1$ points on a line.

\textbf{Exact bound in higher dimensions?} A natural further direction is to obtain an exact bound in \cref{higher-dimensions}, analogous to \cref{no-k+1-in-line}. In order to achieve this using our approach, one would need to come up with a suitable version of the regularisation step (\cref{regularisation}). Note that a set $S' \subseteq [n]^d$ that contains exactly $k$ points in each axis-aligned affine subspace of dimension $t$ corresponds to a multipartite $(n, d, d-t, k)$-design (that is, a $d$-uniform $d$-partite hypergraph with $n$ vertices in each part, such that every set of $d-t$ vertices from different parts is contained in exactly $k$ of its edges). Then, one would need to find such a set $S'$ inside a ``random-like'' configuration $S \subseteq [n]^d$ with $(1+o(1)) k$ points in each axis-aligned affine subspace of dimension $t$. This appears to be related to the challenging problem of finding designs inside Erd\H{o}s--R\'enyi random hypergraphs of appropriate density (see \cite{DKP-24,jain-pham-24,KKKMO-23,keevash-22,SSS-23} for partial results in this direction).

\bibliographystyle{plain}

\bibliography{references}

@inproceedings {jain-pham-24,
    AUTHOR = {Jain, Vishesh and Pham, Huy Tuan},
     TITLE = {Optimal thresholds for {L}atin squares, {S}teiner triple
              systems, and edge colorings},
 BOOKTITLE = {Proceedings of the 2024 {A}nnual {ACM}-{SIAM} {S}ymposium on
              {D}iscrete {A}lgorithms ({SODA})},
     PAGES = {1425--1436},
 PUBLISHER = {SIAM, Philadelphia, PA},
      YEAR = {2024},
      ISBN = {978-1-61197-791-2},
   MRCLASS = {68Q87},
  MRNUMBER = {4699302},
       DOI = {10.1137/1.9781611977912.57},
       URL = {https://doi.org/10.1137/1.9781611977912.57},
}

@article {Rot51,
    AUTHOR = {Roth, K. F.},
     TITLE = {On a problem of {H}eilbronn},
   JOURNAL = {J. London Math. Soc.},
  FJOURNAL = {The Journal of the London Mathematical Society},
    VOLUME = {26},
      YEAR = {1951},
     PAGES = {198--204},
      ISSN = {0024-6107},
   MRCLASS = {10.0X},
  MRNUMBER = {41889},
MRREVIEWER = {P. Scherk},
       DOI = {10.1112/jlms/s1-26.3.198},
       URL = {https://doi.org/10.1112/jlms/s1-26.3.198},
}

@book{dudeney-1917,
  author    = {Dudeney, H. E.},
  title     = {Amusements in Mathematics},
  year      = {1917},
  publisher = {Nelson},
  address   = {London}
}

@book{eppstein-18,
  author    = {Eppstein, David},
  title     = {Forbidden Configurations in Discrete Geometry},
  year      = {2018},
  publisher = {Cambridge University Press}
}

@incollection{brass-moser-pach-05,
  author    = {Brass, Peter and Moser, William O. J. and Pach, J{\'a}nos},
  title     = {Lattice Point Problems},
  booktitle = {Research Problems in Discrete Geometry},
  publisher = {Springer},
  address   = {New York},
  year      = {2005},
  pages     = {417--433},
  doi       = {10.1007/0-387-29929-7_11}
}

@unpublished{lefmann-12,
  author       = {Lefmann, Hanno},
  title        = {Extensions of the No-Three-In-Line Problem},
  year         = {2012},
  note         = {Preprint},
  url          = {https://www.tu-chemnitz.de/informatik/ThIS/downloads/publications/lefmann_no_three_submitted.pdf}
}

@article {HJSW-75,
    AUTHOR = {Hall, R. R. and Jackson, T. H. and Sudbery, A. and Wild, K.},
     TITLE = {Some advances in the no-three-in-line problem},
   JOURNAL = {J. Combinatorial Theory Ser. A},
  FJOURNAL = {Journal of Combinatorial Theory. Series A},
    VOLUME = {18},
      YEAR = {1975},
     PAGES = {336--341},
      ISSN = {0097-3165},
   MRCLASS = {10E99 (05B30)},
  MRNUMBER = {366817},
MRREVIEWER = {Richard\ K.\ Guy},
       DOI = {10.1016/0097-3165(75)90043-6},
       URL = {https://doi.org/10.1016/0097-3165(75)90043-6},
}

@article{guy-kelly-68,
  author  = {Guy, Richard K. and Kelly, Patrick A.},
  title   = {The No-Three-In-Line Problem},
  journal = {Canadian Mathematical Bulletin},
  volume  = {11},
  number  = {4},
  year    = {1968},
  pages   = {527--531},
  doi     = {10.4153/CMB-1968-062-3}
}

@article {KNS-25,
      title={Settling the no-$(k+1)$-in-line problem when $k$ is not small}, 
      author={Benedek Kovács and Zoltán Lóránt Nagy and Dávid R. Szabó},
      year={2025},
      eprint={2502.00176v1},
      archivePrefix={arXiv},
      primaryClass={math.CO},
      url={https://arxiv.org/abs/2502.00176v1},
      note={Preprint, arXiv:2502.00176v1},
}

@article {KNS-25-algebraic,
      title={Randomised algebraic constructions for the no-$(k+1)$-in-line problem}, 
      author={Benedek Kovács and Zoltán Lóránt Nagy and Dávid R. Szabó},
      year={2025},
      eprint={2508.07632},
      archivePrefix={arXiv},
      primaryClass={math.CO},
      url={https://arxiv.org/abs/2508.07632}, 
      note={Preprint, arXiv:2508.07632},
}

@unpublished{green-100-open-problems,
  author = {Green, Ben},
  title  = {100 Open Problems},
  note   = {Manuscript},
  url    = {https://people.maths.ox.ac.uk/greenbj/papers/open-problems.pdf},
  urldate= {2025-09-15}
}

@article {schmidt-68,
    AUTHOR = {Schmidt, Wolfgang M.},
     TITLE = {Asymptotic formulae for point lattices of bounded determinant
              and subspaces of bounded height},
   JOURNAL = {Duke Math. J.},
  FJOURNAL = {Duke Mathematical Journal},
    VOLUME = {35},
      YEAR = {1968},
     PAGES = {327--339},
      ISSN = {0012-7094,1547-7398},
   MRCLASS = {10.25},
  MRNUMBER = {224562},
MRREVIEWER = {E.\ S.\ Barnes},
       URL = {http://projecteuclid.org/euclid.dmj/1077377618},
}

@article{brass-knauer-03,
title = {On counting point-hyperplane incidences},
journal = {Computational Geometry},
volume = {25},
number = {1},
pages = {13-20},
year = {2003},
note = {European Workshop on Computational Geometry - CG01},
issn = {0925-7721},
doi = {https://doi.org/10.1016/S0925-7721(02)00127-X},
url = {https://www.sciencedirect.com/science/article/pii/S092577210200127X},
author = {Peter Brass and Christian Knauer}
}

@article {sudakov-tomon-24,
    AUTHOR = {Sudakov, Benny and Tomon, Istv\'an},
     TITLE = {Evasive sets, covering by subspaces, and point-hyperplane
              incidences},
   JOURNAL = {Discrete Comput. Geom.},
  FJOURNAL = {Discrete \& Computational Geometry. An International Journal
              of Mathematics and Computer Science},
    VOLUME = {72},
      YEAR = {2024},
    NUMBER = {3},
     PAGES = {1333--1347},
      ISSN = {0179-5376,1432-0444},
   MRCLASS = {51E20 (52C10 94B05 94B27)},
  MRNUMBER = {4804978},
       DOI = {10.1007/s00454-023-00601-1},
       URL = {https://doi.org/10.1007/s00454-023-00601-1},
}

@article {ghosal-goenka-keevash-25,
      title={On subsets of lattice cubes avoiding affine and spherical degeneracies}, 
      author={Anubhab Ghosal and Ritesh Goenka and Peter Keevash},
      year={2025},
      eprint={2509.06935},
      archivePrefix={arXiv},
      primaryClass={math.CO},
      url={https://arxiv.org/abs/2509.06935},
      note={Preprint, arXiv:2509.06935}, 
}

@article {HSS-11,
    AUTHOR = {Haeupler, Bernhard and Saha, Barna and Srinivasan, Aravind},
     TITLE = {New constructive aspects of the {L}ov\'asz local lemma},
   JOURNAL = {J. ACM},
  FJOURNAL = {Journal of the ACM},
    VOLUME = {58},
      YEAR = {2011},
    NUMBER = {6},
     PAGES = {Art. 28, 28},
      ISSN = {0004-5411,1557-735X},
   MRCLASS = {68Q87 (05D40 68R05 68W10 68W20 68W25 68W40)},
  MRNUMBER = {2863399},
MRREVIEWER = {Jochen\ Messner},
       DOI = {10.1145/2049697.2049702},
       URL = {https://doi.org/10.1145/2049697.2049702},
}

@book{alon-spencer-16,
  author    = {Noga Alon and Joel H. Spencer},
  title     = {The Probabilistic Method},
  edition   = {4th},
  series    = {Wiley Series in Discrete Mathematics and Optimization},
  publisher = {John Wiley \& Sons, Inc.},
  address   = {Hoboken, NJ},
  year      = {2016}
}

@article {flammenkamp-92,
    AUTHOR = {Flammenkamp, Achim},
     TITLE = {Progress in the no-three-in-line problem},
   JOURNAL = {J. Combin. Theory Ser. A},
  FJOURNAL = {Journal of Combinatorial Theory. Series A},
    VOLUME = {60},
      YEAR = {1992},
    NUMBER = {2},
     PAGES = {305--311},
      ISSN = {0097-3165,1096-0899},
   MRCLASS = {05B99},
  MRNUMBER = {1168162},
MRREVIEWER = {Glenn\ H.\ Hurlbert},
       DOI = {10.1016/0097-3165(92)90012-J},
       URL = {https://doi.org/10.1016/0097-3165(92)90012-J},
}

@article {flammenkamp-98,
    AUTHOR = {Flammenkamp, Achim},
     TITLE = {Progress in the no-three-in-line problem. {II}},
   JOURNAL = {J. Combin. Theory Ser. A},
  FJOURNAL = {Journal of Combinatorial Theory. Series A},
    VOLUME = {81},
      YEAR = {1998},
    NUMBER = {1},
     PAGES = {108--113},
      ISSN = {0097-3165,1096-0899},
   MRCLASS = {05B99},
  MRNUMBER = {1492871},
       DOI = {10.1006/jcta.1997.2829},
       URL = {https://doi.org/10.1006/jcta.1997.2829},
}

@article{anderson-79,
title = {Update on the no-three-in-line problem},
journal = {Journal of Combinatorial Theory, Series A},
volume = {27},
number = {3},
pages = {365-366},
year = {1979},
issn = {0097-3165},
doi = {https://doi.org/10.1016/0097-3165(79)90025-6},
url = {https://www.sciencedirect.com/science/article/pii/0097316579900256},
author = {David Brent Anderson}
}

@misc{DKP-24,
      title={Thresholds for $(n,q,2)$-{S}teiner Systems via Refined Absorption}, 
      author={Michelle Delcourt and Tom Kelly and Luke Postle},
      year={2024},
      eprint={2402.17858},
      archivePrefix={arXiv},
      primaryClass={math.CO},
      url={https://arxiv.org/abs/2402.17858}, 
      note={Preprint, arXiv:2402.17858},
}

@book{lovasz-book,
  author    = {Lov{\'a}sz, L.},
  title     = {Combinatorial Problems and Exercises},
  edition   = {2nd},
  publisher = {AMS Chelsea Publishing},
  address   = {Providence, RI},
  year      = {2007}
}

@article{KKKMO-23,
    AUTHOR = {Kang, Dong Yeap and Kelly, Tom and K\"uhn, Daniela and
              Methuku, Abhishek and Osthus, Deryk},
     TITLE = {Thresholds for {L}atin squares and {S}teiner triple systems:
              bounds within a logarithmic factor},
   JOURNAL = {Trans. Amer. Math. Soc.},
  FJOURNAL = {Transactions of the American Mathematical Society},
    VOLUME = {376},
      YEAR = {2023},
    NUMBER = {9},
     PAGES = {6623--6662},
      ISSN = {0002-9947,1088-6850},
   MRCLASS = {05B15 (05B05 05B07 05C80)},
  MRNUMBER = {4630786},
MRREVIEWER = {Carl\ Johan\ Casselgren},
       DOI = {10.1090/tran/8954},
       URL = {https://doi.org/10.1090/tran/8954},
}

@article {SSS-23,
    AUTHOR = {Sah, Ashwin and Sawhney, Mehtaab and Simkin, Michael},
     TITLE = {Threshold for {S}teiner triple systems},
   JOURNAL = {Geom. Funct. Anal.},
  FJOURNAL = {Geometric and Functional Analysis},
    VOLUME = {33},
      YEAR = {2023},
    NUMBER = {4},
     PAGES = {1141--1172},
      ISSN = {1016-443X,1420-8970},
   MRCLASS = {05B07 (60C05)},
  MRNUMBER = {4616696},
MRREVIEWER = {Luc\ Teirlinck},
       DOI = {10.1007/s00039-023-00639-6},
       URL = {https://doi.org/10.1007/s00039-023-00639-6},
}

@misc{keevash-22,
      title={The optimal edge-colouring threshold}, 
      author={Peter Keevash},
      year={2022},
      eprint={2212.04397},
      archivePrefix={arXiv},
      primaryClass={math.CO},
      url={https://arxiv.org/abs/2212.04397}, 
      note={Preprint, arXiv:2212.04397},
}

@misc{eppstein-blog-18,
  author    = {Eppstein, David},
  title     = {Random no-three-in-line sets},
  note      = {Blog post},
  year      = {2018},
}

@misc{BIM,
      title={A proof of the {K}im-{V}u sandwich conjecture}, 
      author={Natalie Behague and Daniel Il'kovič and Richard Montgomery},
      year={2025},
      eprint={2510.20765},
      archivePrefix={arXiv},
      primaryClass={math.CO},
      url={https://arxiv.org/abs/2510.20765}, 
      note={Preprint, arXiv:2510.20765},
}

@article {KRRS-23,
    AUTHOR = {Klimo\v{s}ov\'a, Tereza and Reiher, Christian and Ruci\'nski,
              Andrzej and \v{S}ileikis, Matas},
     TITLE = {Sandwiching biregular random graphs},
   JOURNAL = {Combin. Probab. Comput.},
  FJOURNAL = {Combinatorics, Probability and Computing},
    VOLUME = {32},
      YEAR = {2023},
    NUMBER = {1},
     PAGES = {1--44},
      ISSN = {0963-5483,1469-2163},
   MRCLASS = {05C80 (05C30 60C05)},
  MRNUMBER = {4523855},
MRREVIEWER = {Mark\ R.\ Jerrum},
       DOI = {10.1017/s0963548322000049},
       URL = {https://doi.org/10.1017/s0963548322000049},
}

@article {KV-04,
    AUTHOR = {Kim, J. H. and Vu, V. H.},
     TITLE = {Sandwiching random graphs: universality between random graph
              models},
   JOURNAL = {Adv. Math.},
  FJOURNAL = {Advances in Mathematics},
    VOLUME = {188},
      YEAR = {2004},
    NUMBER = {2},
     PAGES = {444--469},
      ISSN = {0001-8708,1090-2082},
   MRCLASS = {05C80 (60C05)},
  MRNUMBER = {2087234},
MRREVIEWER = {Anthony\ Bonato},
       DOI = {10.1016/j.aim.2003.10.007},
       URL = {https://doi.org/10.1016/j.aim.2003.10.007},
}

@misc{GIM-20,
      title={{K}im--{V}u's sandwich conjecture is true for $d \gg \log^4 n$}, 
      author={Pu Gao and Mikhail Isaev and Brendan McKay},
      year={2020},
      eprint={2011.09449},
      archivePrefix={arXiv},
      primaryClass={math.CO},
      url={https://arxiv.org/abs/2011.09449},  
      note={Preprint, arXiv:2011.09449},
}

@misc{flammenkamp-website,
  author  = {Achim Flammenkamp},
  title   = {The No-Three-in-Line Problem},
  note     = {\url{https://wwwhomes.uni-bielefeld.de/achim/no3in/readme.html}},
  urldate = {2026-09-12}
}

@inproceedings{dvir-lovett,
  title={Subspace evasive sets},
  author={Dvir, Zeev and Lovett, Shachar},
  booktitle={Proceedings of the forty-fourth annual ACM symposium on Theory of computing},
  pages={351--358},
  year={2012}
}

@misc{conlon,
  author       = {David Conlon},
  howpublished = {Private communication},
  year         = {},
  month        = {},
  note         = {}
}

\end{document}